\documentclass[12pt]{amsart}

\usepackage{amssymb, a4wide, mathdots,url,hyperref}
\usepackage{braket}
\usepackage{verbatim}
\usepackage[usenames]{xcolor}
\usepackage[all]{xy}
\usepackage{paralist}

\usepackage[normalem]{ulem}
\newcommand{\stkout}[1]{\ifmmode\text{\sout{\ensuremath{#1}}}\else\sout{#1}\fi}

\usepackage{cleveref}

\long\def\delete#1{}

\newtheorem{theorem}{Theorem}[section]
\newtheorem{lemma}[theorem]{Lemma}
\newtheorem{proposition}[theorem]{Proposition}
\newtheorem{prop}[theorem]{Proposition}
\newtheorem{remark}[theorem]{Remark}
\newtheorem{conjecture}[theorem]{Conjecture}
\newtheorem{definition}[theorem]{Definition}
\newtheorem{defn}[theorem]{Definition}
\newtheorem{stat}[theorem]{Statement}
\newtheorem{corollary}[theorem]{Corollary}

\newtheorem{claim}{Claim}

\newcommand{\abs}[1]{\lvert#1\rvert}

\newcommand{\Sp}{\mathrm{Sp}\,}

\newcommand{\supp}{\mathrm{supp}\,}

\newcommand{\jm}{{\bf r}}
\newcommand{\ip}{{\bf i}}

\DeclareMathOperator{\Gal}{Gal}
\DeclareMathOperator{\rec}{rec}
\DeclareMathOperator{\SL}{SL}
\DeclareMathOperator{\Irr}{Irr}
\DeclareMathOperator{\Cusp}{Cusp}
\def\IrrE{\Irr_E}
\def\CuspE{\Cusp_E}
\def\Res{{\rm Res}}
\def\GL{{\rm GL}}
\def\rU{\mathrm{U}}
\def\fJ{\mathfrak{J}}

\def\fmm{{\mathfrak{m}}}
\def\bfU{\mathbf{U}}
\def\Alg{\Pi}
\def\Hom{{\rm Hom}}

\newcommand{\la}{{\lambda}}

\newcommand{\cD}{{\mathcal D}}

\newcommand{\cO}{{\mathcal O}}

\newcommand{\Z}{{\mathbb Z}}

\newcommand{\bC}{\mathbb{C}}

\newcommand{\gotM}{\mathfrak{m}}

\title [On two questions concerning distinction]{On two questions concerning representations distinguished by the Galois involution}

\author{Maxim Gurevich}
\address{Department of Mathematics, Weizmann Institute of Science, Rehovot 7610001, Israel}
\email{max.gurevich@weizmann.ac.il}
\author{Jia-Jun Ma}
\address{Department of Mathematics, National University of Singapore, Singapore-119076}
\email{matmjj@nus.edu.sg}
\author{Arnab Mitra}
\address{Department of Mathematics, Technion - Israel Institute of Technology , Haifa 3200003, Israel}
\email{00.arnab.mitra@gmail.com}

\thanks{Maxim Gurevich, partially supported by the ISF grant 756/12, and ERC StG grant 637912.}
\thanks{Arnab Mitra, partially supported by postdoctoral fellowships funded by the Skirball Foundation via the Center for Advanced Studies in Mathematics at Ben-Gurion University of the Negev and the Department of Mathematics, Technion.}
\date{\today}

\numberwithin{equation}{section}

\begin{document}

\setcounter{tocdepth}{1}
\date{\today}

\begin{abstract}
Let $E/F$ be a quadratic extension of non-archimedean local fields of characteristic $0$. In this paper, we investigate two approaches which attempt to describe the irreducible smooth representations of $\GL_n(E)$ that are distinguished by its subgroup $\GL_n(F)$. One relates this class to representations which come as base change lifts from a quasi-split unitary group $F$, while another deals with a certain symmetry condition. By characterizing the union of images of the base change maps we show that these two approaches are closely related. Using this observation, we are able to prove a statement relating base change and distinction for ladder representations. We then produce a wide family of examples in which the symmetry condition does not impose $\GL_n(F)$-distinction, and thus exhibit the limitations of these two approaches.
\end{abstract}

\maketitle
\tableofcontents
\section{Introduction}

Let $F$ be a non-archimedean local field of characteristic
  0. Let $G$ be the group of $F$-points of a reductive linear algebraic
  group. Let $H<G$ be a closed subgroup. A smooth, $\bC$-representation $(\pi,V)$ of $G$ is called $H$-distinguished if there exists a
  non-zero linear functional $\ell$ on $V$ such that
  $\ell(\pi(h)v)=\ell(v)$ for all $h\in H$ and $v\in V$.

  The class of distinguished representations plays a central role in harmonic
  analysis of homogeneous spaces. Furthermore, distinguished representations
  were shown to be crucial for the global theory of period integrals of
  automorphic forms, and have applications to the study of special values of
  $L$-functions.

  This is part of the motivation for a systematic study of $H$-distinguished
  representations of $G$, for a given pair $(G,H)$ of interest.  In particular,
  a highly desirable goal is to obtain a precise classification of
  $H$-distinguished irreducible representations of $G$.

  Applications as mentioned above are often more relevant in cases when the
  homogeneous space $G/H$ is symmetric. In other words, $H$ is a subgroup of
  fixed points of an involution on $G$. In this article we will consider the
  following specific family of symmetric pairs. Let $E$ be a quadratic extension
  of $F$, and let $\tau\in {\rm Gal}(E/F)$ be the non-trivial involution. We identify $\GL_n(F)$ with the subgroup of $\GL_n(E)$ fixed by
  $\tau$.  

\subsection{Conjectures}\label{sec:intro.conj}

Denote by $\omega$ the quadratic character of $F^{\times}$ whose kernel is the
image of the norm map from $E^{\times}$ to $F^{\times}$. We fix a character $\chi_{-1}$ of $E^\times$ extending $\omega$ from now on. Let $\pi^{\tau}$ be the representation of $\GL_{n}(E)$ given by
\begin{equation}\label{eq:pi_tau}
\pi^{\tau}(g)=\pi(\tau(g))
\end{equation}
for $g \in \GL_{n}(E)$.

\subsubsection{} One perspective attempts to characterize irreducible distinguished
representations in this case in terms of a certain symmetry condition. Let us
write here one version of a precise statement of this perspective, namely, the
version which appeared in \cite{K}.

\begin{stat}\label{jac}
  Suppose that $\pi$ is an irreducible admissible representation of
  $\GL_{n}(E)$ such that the central character of $\pi$ is trivial on
  $F^{\times}$.
\begin{enumerate}
\item Suppose $n$ is odd. Then $\pi^{\vee}=\pi^{\tau}$ if and only if $\pi$ is $\GL_{n}(F)$-distinguished.
\item Suppose $n$ is even. Then $\pi^{\vee}=\pi^{\tau}$ if and only if either $\pi$ or $(\chi_{-1}) \pi$ is $\GL_{n}(F)$-distinguished.
\end{enumerate}
\end{stat}

There have been works proving the statement for several classes of
representations in the past. The reader may see, for example \cite{P},
\cite{HM}, \cite{K}, \cite{M}, \cite{M1}. However \Cref{jac} is incorrect in general. For a counterexample, consider the irreducible representation ${\rm Ind}^{{\GL}_{3}(E)}_{P}(1\otimes (\chi_{-1})\sigma)$ where $1$ is the trivial character of $E^{\times}$, $\sigma$ is an irreducible $\GL_{2}(F)$-distinguished cuspidal representation of $\GL_{2}(E)$, and $P$ is
the standard parabolic subgroup of $\GL_{3}(E)$ with Levi subgroup $\GL_{1}(E)\times \GL_{2}(E)$.

We call an irreducible representation of $\GL_{n}(E)$ {\it rigid} if its
cuspidal support is contained in a single cuspidal line (see \S \ref{def:cusp_supp}). Motivated by the above example, let us restrict ourselves to the class
of rigid representations and consider a weaker version of the above statement
(as formulated in \cite{G}):
\begin{conjecture}\label{jac2}
Suppose that $\pi$ is a rigid representation of $\GL_{n}(E)$. Then $\pi^{\vee}=\pi^{\tau}$ if and only if at least one of $\pi$ and $\chi_{-1}\pi$ is $\GL_{n}(F)$-distinguished.
\end{conjecture}

Conjecture \ref{jac2} was resolved by the first author in \cite{G} for ladder representations which is a subclass of rigid representations recently introduced by Lapid and M\'inguez (\cite{LM}). The class of ladder representations contains all Speh
representations. Since Speh representations are the building blocks of the
unitary dual (see \cite{Ta}), any irreducible unitarizable representation of
$\GL_{n}(E)$ is a product of ladder representations.

\subsubsection{} Yet another approach for studying distinguished representation is through Langlands functorial lifts. A paradigm, motivated mostly by the relative trace formula approach for automorphic representations, is that the class of $H$-distinguished representations of $G$ may correspond to the image of a certain functorial lift.

Let $\rU_n$ be the quasi-split unitary group over $n$ variables defined with respect to the quadratic extension $E/F$. The {\it stable} and the {\it unstable} base change maps take irreducible admissible representations of $\rU_{n}$ to those of $\GL_{n}(E)$. Their image, on the
Langlands parameter level, is classified in \cite{GGP} (see also \cite[Lemma~2.2.1]{Mo}). 
We have the following statement (see \cite{F}, \cite{F1}) relating the
image of the base change maps with the class of irreducible $H$-distinguished
representations of $G$. We state below the version in \cite[Conjecture~1.1]{AR}.
\begin{stat}\label{rafli}
  Let $\pi$ be an irreducible representation of $\GL_{n}(E)$. If $n$ is odd
  (resp., even), then $\pi$ is $\GL_{n}(F)$-distinguished if and only if it is a
  stable (resp., unstable) base change from $\rU_{n}$.
\end{stat}

For cuspidal representations of $\GL_{3}(E)$, \Cref{rafli} was proved in \cite{AR}. On the other hand, it is not valid in general as the following example indicates. Let $\pi=L([\nu^{-1}\sigma,\sigma][\sigma,\nu\sigma])$ (see \S \ref{s_irr_gln} for the notation) where $\sigma$ is an arbitrary conjugate self-dual cuspidal representation of $\GL_{k}(E)$ for some $k$. By the results in \cite{Mo} and
\cite{G} (Lemma \ref{mok_res} and Theorem \ref{j_ladder} in this article), it is
easy to see that for an appropriate choice of $\sigma$, the representation $\pi$
lies in the image of the unstable base change map but is not
$H$-distinguished. However, $\chi_{-1}\pi$ is indeed so. 

Inspired by this example and \Cref{jac2}, we propose a weaker version of \Cref{rafli}:
\begin{conjecture}\label{rafli2}
Let $\pi$ be a rigid representation of $\GL_{n}(E)$. Then at least one of $\pi$ and $\chi_{-1}\pi$ is $\GL_{n}(F)$-distinguished if and only if $\pi$ lies in the union of the images of the stable base change map and the unstable base change map.
\end{conjecture}

\subsection{Main Results}
We will now describe, in short, the main results that we obtain in this article. To ease the burden of notation in the exposition, we do not provide them here in their most general form. 

\subsubsection{}
 We begin by obtaining an explicit combinatorial description of the rigid representations that lie in the union of the images of the stable and the unstable base change maps:
\begin{prop}[\Cref{equiv}]\label{equiv_int}
Let $\pi$ be a rigid representation of $\GL_{n}(E)$. Then the following two statements are equivalent:
\begin{enumerate}
\item $\pi^{\vee}\cong \pi^{\tau}$
\item $\pi$ lies in the union of the images of the stable base change map and
  the unstable base change map
\end{enumerate}
\end{prop}

Note that the statement of \Cref{equiv_int} is analogous the description of the image of the base change map from irreducible representations of $\GL_{n}(F)$ to those of $\GL_{n}(E)$ as provided by Arthur and Clozel (cf. \cite[Theorem 6.2(a),(b)]{AC}).    

\subsubsection{}
\Cref{equiv_int} has several interesting consequences. An immediate one is the
equivalence of Conjecture \ref{jac2} and Conjecture \ref{rafli2}. Another one is
that the set of rigid representations of $\GL_{n}(E)$ contained in the union of
the two base change maps is invariant under the Zelevinsky involution (see \S \ref{ss_bc_zel}).

Using the equivalence of the two conjectures and the results of \cite{G}, we describe the relation of the base change maps with distinction for ladder representations (see \Cref{def_lad}):
\begin{theorem}[\Cref{rf_ladder}]
  Let $\pi$ be a ladder representation of $\GL_{n}(E)$ for some $n$. Then at
  least one of $\pi$ and $\chi_{-1}\pi$ is $\GL_{n}(F)$-distinguished if and
  only if $\pi$ lies in the union of the images of the two base change maps.
\end{theorem}

\subsubsection{}
Finally, for the case of mutually unlinked ladder representations (see \Cref{def_mut_unl}), we obtain a converse to the hereditary property of $\GL_{n}(F)$-distinguished representations:
\begin{prop}[\Cref{conv_her_prop}]\label{conv_her_prop_int}
Let $\pi$ be an irreducible representation of $\GL_{n}(E)$ and suppose that there exists mutually unlinked proper ladder representations $\pi_1,\ldots,\pi_k$ such that
\begin{equation*}
\pi= {\rm Ind}^{GL_{n}(E)}_{P}(\pi_1\otimes\cdots\otimes \pi_k)
\end{equation*}
and $P$ is the parabolic corresponding to $\pi_1\otimes\cdots\otimes \pi_k$. Then $\pi$ is $\GL_{n}(F)$-distinguished if and only if there is a permutation $w$ in $k$ variables such that:
\begin{enumerate}
\item $\pi_{w(i)}\cong (\pi_i^\vee)^{\tau}$ for all $1\leq i\leq k$
\item For every $i$ such that $w(i)=i$, $\pi_i$ is $\GL_{n_{i}}(F)$-distinguished (for the appropriate integer $n_{i}$).
\end{enumerate}
\end{prop}

Note that by \cite[Theorem 9.7]{Z} the generic irreducible representations of $\GL_{n}(E)$ satisfy the hypotheses of \Cref{conv_her_prop_int}. Thus this result generalizes \cite[Theorem~5.2]{M}.

Another interesting consequence of \Cref{conv_her_prop_int} is that it allows us
to come up with many examples of irreducible representations which satisfy the
symmetry condition in \Cref{jac2}, but neither they nor their twists by
the character $\chi_{-1}$ is $\GL_{n}(F)$-distinguished. Since \Cref{jac2} and
\Cref{rafli2} are equivalent, it demonstrates the failure of the two conjectures
in the class of representations irreducibly induced from ladders. We finish
the paper by giving yet another counterexample, consisting of an \emph{imprimitive} representation of $\GL_{n}(E)$ (see \S\ref{sec_imp_ce}).

\subsection{Structure of the Paper}
Let us now delineate the contents of this paper. After briefly setting up some
general notation in \S \ref{s_gen_not}, we move on \S \ref{s_irr_gln} where we
recall the definition of segments and other preliminaries concerning the
classification of irreducible representations of $\GL_{n}$ over a
non-archimedean local field a la Zelevinsky. After this, in \S \ref{s_pri_gal}
we recall the basic definitions and some preliminary notions of $L$-groups and
$L$-parameters, and the two base change lifts that we need for this article. In \S
\ref{s_equiv} we obtain a characterization of the union of the image of the two
base change maps for rigid representations and derive some consequences of
it. The results for the ladder representations, in particular the validity of
Conjecture \ref{rafli2} for this class, are contained in \S
\ref{s_dist_bc_ladder}. In \S \ref{s_ind_ladder}, we provide a converse
to the hereditary property of $H$-distinguished representations for the case of
mutually unlinked ladders. Finally, in \S \ref{sec_imp_ce} we provide an example of an imprimitive representation of $\GL_{n}(E)$ which satisfies the
symmetry condition in \Cref{jac2}, but neither it nor its twist by the character $\chi_{-1}$ is $\GL_{n}(F)$-distinguished.

\subsection{Acknowledgments}
The authors would like to thank Wee Teck Gan, Erez Lapid, Omer Offen, Dipendra Prasad, Eitan Sayag, and Jiu Kang Yu for several helpful conversations. The third author would like to thank Steven Spallone for answering his questions on parity of self-dual representations and sharing his preprint on the subject matter with him. The second and the third author would like to thank the Hausdorff Institute for Mathematics (Bonn) for its warm hospitality where this project was initiated. Part of the work was done during the third author's visit to CUHK (Hong Kong). It is a pleasure for him to thank Jiu Kang Yu for inviting him for the same and the institute for providing an excellent work environment.  

\section{General notation}\label{s_gen_not}
We set up some primary notation in this section. More specific notation is defined in the sections when it first occur.
\subsubsection{}
We will closely follow here the notation of \cite{Mo} when applicable. Let $E/F$
be a fixed quadratic extension of non-archimedean local fields of characteristic
$0$. We will often use bold font (for example $\mathbf{G}, \mathbf{H}$) to
denote an algebraic group defined over $F$ and usual font (for example $G, H$) to denote the
topological group of its $F$-points. Let $\Alg(G)$ denote the category of
complex valued, smooth, representations of $G$ of finite length and $\Irr(G)$
the class of irreducible representations in $\Alg(G)$.

Let $\pi^{\vee}$ be the contragredient of a representation $\pi \in
\Alg(G)$.

\subsubsection{}We will refer to distinction of representations in the following sense:
\begin{defn}
Let $G'$ be a closed subgroup of $G$ and let $\chi$ be a character of $G'$. We say that $(\pi,V_{\pi})\in \Alg (G)$ is $(G',\chi)$-distinguished if there exists a non-zero $G'$-invariant linear functional on $V_{\pi}$ which transforms by the character $\chi$, i.e.
$$\Hom_{G'}(V_{\pi}|_{G'},\chi)\neq 0.$$
\end{defn}
If $\chi$ is the trivial character of $G'$, we simply say that $\pi$ is $G'$-distinguished.

\subsubsection{} Henceforth, $\mathbf{G_{n}}$ and $\mathbf{H_{n}}$ will always
denote the reductive groups $\mathbf{\Res_{E/F}\GL_{n}}$ and $\mathbf{\GL_{n}}$
respectively. Thus $G_{n}\cong \GL_{n}(E)$. 

Recall that, in \S\ref{sec:intro.conj}, we set $\omega=\omega_{E/F}$
to be the quadratic character of $F^{\times}$ whose kernel is equal to the image of
the norm map from $E^{\times}$ to $F^{\times}$. We fixed an extension
  $\chi_{-1}$ of $\omega$ to $E^\times$. Furthermore, $\omega$ and $\chi_{-1}$ are also
  viewed as characters of $W_F$ and $W_E$ respectively via the reciprocity map of local class field
  theory where $W_{F}$ and $W_{E}$ are the Weil groups of the respective fields.  

Denote by $\tau$ the non-trivial element of ${\rm Gal}(E/F)$. For
$\pi\in \Alg(G_{n})$, denote by $\pi^{\tau}$ the representation of $G_{n}$ on
the space of $\pi$ given by $\pi^{\tau}(g)=\pi(\tau(g))$.

The norm character of $G_{n}$ (resp. $H_{n}$) given by $|{\rm det}(\cdot)|_{E}$
(resp.  $|{\rm det}(\cdot)|_{F}$) is denoted by $\nu_{E}$ (resp. $\nu_{F}$). We
will sometimes suppress the field in the subscript of the character if it is
clear from the context.

\subsubsection{} To shorten our notation, for $\pi\in \Irr(G_{n})$, we will often omit the subscript and use the phrase $\pi$ is $H$-distinguished (resp. $(H,\omega)$-distinguished) to say that $\pi$ is $H_{n}$-distinguished (resp. $(H_{n},\omega)$-distinguished), for the corresponding $n$.

\subsubsection{}
Set $G=G_{n}$ and let $P=M\ltimes U$ be a standard parabolic subgroup of $G$
with its standard  Levi decomposition. We will denote by $\ip_{G,M}$ the
normalized parabolic induction functor from $\Alg(M)$ to $\Alg(G)$. Let $(n_{1},\dots,n_{k})$ be a decomposition of $n$ and let $\sigma_i\in\Alg(G_{n_i})$, $i=1,\dots,k$.  Assume that $M\cong \Pi_{i=1}^{k}G_{n_i}$. Then $\sigma=\sigma_1\otimes \cdots \otimes \sigma_k\in \Alg(M)$. Set
\begin{equation*}
\sigma_1\times\cdots \times \sigma_k :=\ip_{G,M}(\sigma).
\end{equation*}

The normalized Jacquet functor $\Alg(G)$ to $\Alg(M)$ will be denoted by $\jm_{M,G}$. It is left adjoint to the normalized parabolic induction functor.

\subsubsection{The quasi-split unitary group}
Denote by $\bfU_{n}=\bfU_{n,E/F}$ the quasi-split unitary group in $n$ variables with respect to the extension $E/F$. The group of $F$-points of $\bfU_{n}$ as matrices is given by
\[
\rU_{n}= \Set{g\in \GL_{n}(E)| {}^{t}\tau(g) \,J_{n}\, g=J_{n}},
\]
where
\begin{equation}\label{def_jn}
J_{n}=\begin{pmatrix}
            & & & 1 \\
            & & -1 &  \\
            & \iddots & & \\
     (-1)^{n-1} & & &
   \end{pmatrix}.
 \end{equation}

 \subsubsection{} We will use the term \emph{multi-set} to mean set with multiplicities. More formally, a multi-set on a (possibly infinite) set $\mathcal D$ means a function from the set $\cD$ to the set of non-negative integers. In particular, for two  multi-sets $\fmm_{1}$ and $\fmm_{2}$, $\fmm_1+\fmm_2$ and $\fmm_1\geq\fmm_2$ make sense. All the multi-sets that we will encounter in this note will be finitely supported. We will denote by $\abs{\fmm}$ the non-negative integer $\Sigma_{x\in \cD}\fmm(x)$.

\subsubsection{} The permutation group on $t$ variables will be denoted by $\mathcal S_{t}$.

\section{Preliminaries on irreducible representations of $\GL_n$}\label{s_irr_gln} 
In this section, let $F'$ denote any non-archimedean local field. Let $\Irr_{F'}$ be the set $\sqcup_{n\geq 0}\Irr(\GL_{n}(F'))$. Denote the set of cuspidal representations in $\Irr_{F'}$ by $\Cusp_{F'}$. For a $\sigma\in \Cusp_{F'}$ define its \emph{cuspidal line}
\begin{equation*}
\sigma^\Z=\set{\nu_{F'}^{m}\sigma| m\in \mathbb Z}.
\end{equation*}

\subsection{Classifications of irreducible representations of $\GL_n(F')$} We now recall the combinatorial notion of {\it segments} introduced by Zelevinsky (in \cite{Z}), and briefly review the classification of irreducible representations of $\GL_{n}(F')$.
\subsubsection{}
\begin{defn}
  Given a representation $\sigma\in \Cusp_{F'}$ and $a,b \in \Z$ such that
  $a\leq b+1$, define the segment $[a,b]_{(\sigma)}$ (also denoted by
    $[\nu^a\sigma,\nu^b\sigma]$) to be the set $\{ \nu^{a}\sigma, \nu^{a+1}\sigma,\dots,\nu^{b}\sigma \}$ if $a\leq b$ and
  the empty set if $a=b+1$. We say that the segment $[a,b]_{(\sigma)}$ is
  supported on $\sigma^\Z$.
\end{defn}

For a segment $\Delta=[a,b]_{(\sigma)}$, we denote by $b(\Delta)=\nu^{a}\sigma$ its beginning, by $e(\Delta)=\nu^{b}\sigma$ its end, and by $\ell(\Delta)=b-a+1$ its length respectively. For a segment $[a,b]_{(\sigma)}$, $[a,b]_{(\sigma)}^{\vee}$ will mean the segment $[-b,-a]_{(\sigma^{\vee})}$ and $[a,b]_{(\sigma)}^{\tau}$ will denote the segment $[a,b]_{(\sigma^{\tau})}$.

The representation $\nu^{a}\sigma \times \cdots\times \nu^{b}\sigma$ has a
unique irreducible subrepresentation and a unique irreducible quotient which we
write as $Z(\Delta)$ and $L(\Delta)$ respectively. By convention, if the set
$\Delta$ is empty, then both $Z(\Delta)$ and $L(\Delta)$ are defined to be the
trivial representation of the trivial group.
\subsubsection{}
\begin{defn}
Two segments $\Delta_{1}$ and $\Delta_{2}$ are said to be linked if $\Delta_{1}\nsubseteq \Delta_{2}$, $\Delta_{2}\nsubseteq \Delta_{1}$ and $\Delta_{1}\cup \Delta_{2}$ is also a segment. If $\Delta_{1}$ and $\Delta_{2}$ are linked and $b(\Delta_{1}\cup\Delta_{2})=b(\Delta_{1})$, then we say that $\Delta_{1}$ precedes $\Delta_{2}$ and write $\Delta_{1}\prec \Delta_{2}$.
\end{defn}
Let $\cO$ be the set of multi-sets of segments. An ordering $\Delta_1, \Delta_2, \cdots, \Delta_t$ on a multi-set
  $\fmm=\set{\Delta_1,\dots,\Delta_t}\in\cO$  is of \emph{standard form} if
$\Delta_i\not \prec\Delta_j$ for all $i<j$. Clearly every $\fmm\in \cO$ admits a
standard order.

\subsubsection{The Zelevinsky classification} Let
$\fmm=\{\Delta_1,\dots,\Delta_t\}\in\cO$ be ordered in standard form. The
representation  
\begin{equation*}
Z(\Delta_1) \times\cdots \times Z(\Delta_t)
\end{equation*}
is independent of the choice of order of standard form. It has a unique irreducible submodule that we denote by $Z(\fmm)$.

The Zelevinsky classification says that the map $(\fmm\mapsto Z(\fmm)):\cO\rightarrow \Irr_{F'}$ is a bijection.

\subsubsection{The Langlands classification} Let
$\fmm=\{\Delta_1,\dots,\Delta_t\}\in\cO$ be ordered in standard form. The
representation
\begin{equation*}
L(\Delta_1)\times\cdots\times L(\Delta_t)
\end{equation*}
is independent of the choice of order of standard form. It has a unique irreducible quotient that we denote by $L(\fmm)$.

The Langlands classification says that the map $( \fmm\mapsto L(\fmm)) :\cO\rightarrow \Irr_{F'}$ is a bijection.

\subsubsection{} For a multi-set $\fmm=\{\Delta_{1},\dots,\Delta_{t}\}$,
$\fmm^{\vee}$ will denote the multi-set
$\{\Delta_{1}^{\vee},\dots,\Delta_{t}^{\vee}\}$. It is known that
$Z(\fmm)^{\vee}\cong Z(\fmm^{\vee})$ and $L(\fmm)^{\vee}\cong L(\fmm^{\vee})$
(see \cite[Theorem 7.10]{Z}).

The multi-set $\fmm^{\tau}$ will denote the multi-set $\{\Delta_{1}^{\tau},\dots,\Delta_{t}^{\tau}\}$. As above we have $Z(\fmm)^{\tau}\cong Z(\fmm^{\tau})$ and $L(\fmm)^{\tau}\cong L(\fmm^{\tau})$.

\subsubsection{The Zelevinsky involution}\label{mwalgo}
It follows from the two classifications above that for any $\fmm\in \cO$ there
exists a unique $\fmm^t\in \cO$ such that $Z(\fmm)=L(\fmm^t)$. The function
$\fmm\mapsto \fmm^t$ is an involution on $\cO$ known as the Zelevinsky
involution. For $\pi=Z(\fmm)\in \Irr_{F'}$, let $\pi^t=L(\fmm)$. Then
$\pi\mapsto \pi^t$ is the corresponding involution on $\Irr_{F'}$.

Given a multi-set $\fmm$, an algorithm to compute $\fmm^{t}$ is provided in \cite{MW}.

\subsection{The cuspidal support}\label{def:cusp_supp}
For every $\pi\in \Irr_{F'}$ there exist
$\sigma_1,\dots,\sigma_k\in \Cusp_{F'}$, unique up to rearrangement, so that
$\pi$ is isomorphic to a subrepresentation of
$\sigma_1\times \cdots \times \sigma_k$ (see \cite[Proposition 1.10]{Z}). Let
\[
\supp(\pi)=\set{\sigma_i|i=1,\dots,k}
\]
be the support of $\pi$. For $\fmm\in\cO$, let
\[
\supp(\fmm)=\{\sigma\in \Cusp_{F'}: \sigma\in\Delta\text{ for some }\Delta\in
\fmm\}
\]
be the support of $\fmm$ \footnote{The support is often considered as a
  multi-set. In this article though only the underlying set plays a role and
  hence we will treat the support, of both a representation and a multi-set of
  segments, as a set.}.

\subsubsection{} A representation $\pi\in \Irr_{F'}$ is said to be \emph{rigid} if
$\supp(\pi)\subseteq \sigma^\Z$ for some $\sigma\in
\Cusp_{F'}$. Similarly, a multi-set
$\fmm\in \cO$ is called rigid if $\supp(\fmm)\subseteq\sigma^\Z$ for some
$\sigma\in \Cusp_{F'}$.  Let
\begin{equation*}
\cO_{\sigma}=\{\fmm\in\cO:\supp(\fmm)\subseteq \sigma^\Z\}
\end{equation*}
be the set of rigid multi-sets supported on $\sigma^\Z$.

\subsubsection{} When dealing with elements of $\cO_{\sigma}$ for a fixed cuspidal representation $\sigma$, if there is no scope for confusion, we
will often omit $\sigma$ from our notation and consider segments as sets of
integers. Let $\Delta=[\nu^{a}\sigma,\nu^{b}\sigma]$. Then $b(\Delta)$ and
$e(\Delta)$ will denote the integers $a$ and $b$. In particular, for $\Delta$
and $\Delta'=[\nu^{a'}\sigma,\nu^{b'}\sigma]$, we will write
$b(\Delta)\leq b(\Delta')$ or $e(\Delta)\leq e(\Delta')$ to denote $a\leq a'$ or
$b\leq b'$ respectively.

\section{Preliminaries on the Galois side and the base change maps}\label{s_pri_gal}
\subsection{The $L$-groups and the local Langlands parameters}
Let $W_E$ denote the Weil group of $E$ and $W'_E:= W_E\times \SL_{2}(\bC)$ denote the Weil-Deligne group.  Let
$\Phi(\GL_n(E))$ be the set of Langlands parameters of $\GL_n(E)$ so that each element $\rho\in \Phi(\GL_n(E))$
is an (equivalence class of) $n$-dimensional representation of $W'_E$ and 
\[
\Phi_E=
\bigsqcup_n \Phi(\GL_n(E)). 
\]
Let 
\[
\xymatrix{
\rec \colon \Irr_E\ar[r] & \Phi_E
}
\]
be the local Langlands reciprocity map established in \cite{HT} (later also in \cite{H} and \cite{Sc}).

By \cite{Z}, the Langlands reciprocity map is reduced to the cuspidal
case. More precisely, if $\pi = L(\fmm)$ where $\fmm =
\set{[\nu^{a_{1}}\sigma_1,\nu^{b_{1}}\sigma_1], \dots,
  [\nu^{a_{t}}\sigma_t,\nu^{b_{t}}\sigma_t]}$, then 
\[
\rec(\pi)= \bigoplus_i \rec(\nu^{\frac{a_{i}+b_{i}}{2}}\sigma_i)\otimes \Sp(b_{i}-a_{i}+1).
\]
Here $\rec(\nu^{\frac{a_{i}+b_{i}}{2}}\sigma_i)$ is an irreducible $W_E$-representation 
and 
$\Sp(m)$ denotes the unique irreducible $m$-dimensional representation of $\SL_2(\bC)$.


The group $W_{E}$ can be naturally identified with a subgroup of $W_{F}$ and the quotient $W_F/W_E \cong \Gal(E/F)$. For the non-trivial element $\tau\in\Gal(E/F)$, let $w_\tau$ be a fixed lift of $\tau$ in $W_F$.  For any $W'_E$-module $\rho$, let $\rho^\tau(w) := \rho(w_\tau w w_\tau^{-1})$ for each $w\in W'_E$\footnote{Note that the isomorphism class of $\rho^\tau$ is independent of the choice of $w_\tau$.}.We will use several times in this article the following fact (\cite[Lemma~VII.1.6]{HT}):
\begin{theorem}\label{thm:rec_gal}
 Let $\pi^\tau$ be as defined in \eqref{eq:pi_tau}. Then
  \[
    \rec(\pi^\tau) = \rec(\pi)^\tau.
  \]
\end{theorem}

\subsection{The two base change maps}
The stable and the unstable base change maps take an irreducible representation of $U_{n}$ to an irreducible representation of $\GL_{n}(E)$ corresponding to two homomorphisms between the respective $L$-groups. We refer the reader to \cite[\S 2.1 and \S 2.2]{Mo} for details. We will  recall here the results that we require for our purposes. 

As earlier, fix a choice of the lift $w_\tau\in W_F$ of the non trivial element $\tau\in \Gal(E/F)$. 
\begin{defn}\label{def_parity}
The parameter $\rho\in \Phi(\GL_n(E))$ is called conjugate self-dual if
\begin{equation*}
\rho^{\vee}\cong \rho^{\tau}.
\end{equation*}
Suppose $\rho$ is conjugate self-dual and realized on a vector space $V$. Say that
  $\rho$ is of \emph{parity $\eta$ ($\eta=\pm 1$)} if there exists a non-degenerate bilinear form $B(\cdot,\cdot)$ on $V$ satisfying the following for all $x,y\in V$:
\begin{equation*}
B(\rho^{\tau}(w)x,\rho(w)y)=B(x,y),\ B(x,y)=\eta B(y,\rho(w_{\tau}^{2})x).
\end{equation*}
\end{defn}

\begin{remark}
The notion of parity is independent of the choice of $w_{\tau}$.
\end{remark}

\begin{remark}\label{rmk:SD}
By \Cref{thm:rec_gal}, for any $\pi \in \Irr_E$, $\rec(\pi)$ is conjugate self-dual if and only if $\pi^\vee \cong \pi^\tau$. 
\end{remark}

\begin{remark}
For a conjugate self-dual $\rho$ as above, parity neither has to exist nor be unique if it exists. Nevertheless when $\rho$ is irreducible, there exists a unique parity due to Schur's lemma.  
\end{remark}

\begin{remark}\label{rmk:parity_gen}
More generally, suppose that $\rho=\rho_{1}\oplus\cdots\oplus\rho_{t}$ is conjugate self-dual where each $\rho_i$ is an irreducible $W'_E$-representation. Then, for each $i$, $(\rho_i^\tau)^\vee $ is isomorphic to $ \rho_j$ for some $j$. Therefore, $\rho$ can be written as 
\begin{equation}\label{eq:rho_factor}
 \rho = \bigoplus_{1\leq i \leq r}\left(\rho'_{i}\oplus ({\rho'_i}^{\tau})^{\vee}\right)\oplus \bigoplus_{1\leq j \leq s }\rho''_{j}
\end{equation} 
for some non-negative integers $r,s$ and irreducible $W'_E$-representations $\rho'_{i}$, $\rho''_{j}$ such that $\set{\rho''_{j}:1\leq j \leq s}$ is a set of pairwise non-isomorphic conjugate self-dual irreducible representations of $W'_E$. (Note that the integers $r$ and $s$ are determined by $\rho$.) The conjugate self-dual representation $\rho'_{i}\oplus({\rho'_{i}}^{\tau})^{\vee}$ admits bilinear forms with both parities.  Moreover if $s\geq 1$, then each $\rho_{j}''$ appears with odd multiplicity. Hence $\rho$ has a parity $\eta$ if and only if $\rho''_j$ have parity $\eta$ for all $1\leq j\leq s$, and $\rho$ has both parities if and only if $s=0$ (see \cite[\S 4]{GGP}). 
\end{remark}

\subsubsection{Image of the base change maps at the parameter level}
Thanks to \cite[Theorem 8.1]{GGP} (see also \cite[Lemma 2.1]{Mo}), the image of the stable and unstable base change maps for the Langlands parameters can be characterized as follows:

\begin{lemma}\label{mok_res}
The image of the stable (resp., unstable) base change map in $\Phi(\GL_n(E))$ is the set of conjugate self-dual parameters with parity $\eta=(-1)^{n-1}$ (resp., $\eta=(-1)^{n}$). 
\end{lemma}

\begin{corollary}\label{distn}
Let $\pi\in {\rm Irr}_E$. If $\pi$ is in the image of the stable base change (resp., unstable base change) map from $U_{n}$ if and only if $\chi_{-1}\pi$ is in the image of the unstable base change (resp., stable base change) map.
\end{corollary}
\begin{proof}
 Let $\rho=\rec(\pi)$. Note that the representation $\rho$ is conjugate self-dual with a parity if and only if $\chi_{-1}\rho$ is conjugate self-dual with a parity. Moreover the parities differ by the factor $\chi_{-1}(w_{\tau}^{2})=-1$. The statement now follows from Lemma \ref{mok_res}.
\end{proof}

\section{Equivalence of Conjecture \ref{jac2} and Conjecture
  \ref{rafli2}}\label{s_equiv}

\subsection{A characterization of the rigid representations in the image of the base change maps}

\subsubsection{Existence of parity for conjugate self-dual rigid representations} 
\begin{lemma}\label{parity}
 Let $\pi\in \IrrE$ be rigid and let $\rho=\rec(\pi)$. If $\pi^{\vee}\cong\pi^{\tau}$, then $\rho$ is conjugate self-dual with a parity.
\end{lemma}
\begin{proof}
By \Cref{rmk:SD}, $\rho$ is conjugate self-dual. We will show now that the rigidity of $\pi$ implies that $\rho$ has a parity. Let $\supp(\pi)\subseteq (\sigma')^{\Z}$.  Write
 \begin{equation*}
 \pi=L([\nu^{a_{1}'}\sigma',\nu^{b_{1}'}\sigma'], \dots,[\nu^{a_{t}'}\sigma',\nu^{b_{t}'}\sigma'])
 \end{equation*}
where $a_{i}', b_{i}'\in \mathbb Z$. Since $\pi^{\vee}\cong\pi^{\tau}$, $\supp(\pi^\vee) = \supp(\pi^\tau)$ which implies that $(\sigma')^{\tau}=\nu^{a}(\sigma')^{\vee}$ for some integer $a$. Let $\sigma=\nu^{-\frac{a}{2}}\sigma'$ so that $\sigma$ is conjugate self-dual. Let $a_{i}=a_{i}'+\frac{a}{2}$ and $b_{i}=b_{i}'+\frac{a}{2}$ for all $i\in \set{1,\dots, t}$. 
Write $\rho = \bigoplus_i \rho_i $ with 
\[
\rho_i = \rec(\nu^{\frac{a_{i}+b_{i}}{2}}\sigma)\otimes \Sp(b_{i}-a_{i}+1).
\]
Note that $\rho_i$ is conjugate self-dual if and only if $a_i+b_i = 0$, i.e $\rho_i = \rho_0\otimes \Sp(2b_i+1)$ where $\rho_0 = \rec(\sigma)$. Let $\eta_0$ be the parity of $\sigma$. Note that the exponents $b_{i}$ are either all in $\mathbb Z$ or all in $\frac{1}{2}+\mathbb Z$, depending on whether $a$ is even or odd. Therefore all conjugate self-dual irreducible components in $\rho$ have the same parity. \Cref{rmk:parity_gen} now implies that $\rho$ has a parity.
\end{proof}

\subsubsection{}
\begin{prop}\label{equiv}
  For a rigid irreducible representation $\pi$ of $G_{n}$, the following two
  statements are equivalent:
\begin{enumerate}[(i)]
\item $\pi^{\vee}\cong \pi^{\tau}$
\item $\pi$ lies in the union of the images of the stable base change map and
  the unstable base change map
\end{enumerate}
In particular, \Cref{jac2} and \Cref{rafli2} are equivalent for rigid irreducible representations.
\end{prop}
\begin{proof}
  Let $\rho=\rec(\pi)$. By \Cref{mok_res} $\pi$ lies in the union of the images
  of the two base change maps if and only if $\rho$ is a conjugate self-dual
  representation of $W_{E}'$ with a parity. Using \Cref{parity} and the local
  Langlands correspondence for $G_{n}$, we see that this is equivalent to the
  condition $\pi^{\vee}\cong \pi^{\tau}$.
\end{proof}

\subsection{The image of the base change maps and Zelevinsky involution}
\label{ss_bc_zel}
\subsubsection{} Our characterization leads to an interesting application.
\begin{proposition}\label{zi_prop}
Let $\pi\in \IrrE$ be a rigid representation in the union of the image of the two base change maps. Then $\pi^{t}$ is also in the union of the image of the base change maps.
\end{proposition}
\begin{proof}
Let $\pi=L(\fmm)$. Then $\pi^{t}=Z(\fmm)$ and we have
\begin{equation*}
  ((\pi^{t})^{\vee})^{\tau}\cong (Z(\fmm)^{\vee})^{\tau} \cong Z((\fmm^{\vee})^{\tau}) \cong L((\fmm^{\vee})^{\tau})^{t}\cong ((\pi^{\vee})^{\tau})^{t} \cong \pi^{t}.
\end{equation*}
The statement now follows from  Proposition \ref{equiv}.
\end{proof}

\subsubsection{} The image of an individual base change map is not preserved by the Zelevinsky involution as demonstrated in our next remark.
\begin{remark}\label{zi_prop1}
  Let $\pi=L([\nu^{-\frac{1}{2}}][\nu^{\frac{1}{2}}])$. The parameter
  $\rec(\pi)$ is conjugate self-dual with both the parities, and thus by Lemma
  \ref{mok_res}, is in the image of the unstable base change map. On the other
  hand, $\pi^{t}=L([\nu^{-\frac{1}{2}},\nu^{\frac{1}{2}}])$ is in the image of
  only the stable base change map.
\end{remark}

\subsubsection{} In view of Remark \ref{zi_prop1}, we can make the following
weaker statement:
\begin{proposition}\label{zi_prop2}
 Let $\pi\in \IrrE$ be a rigid representation in the image of exactly one of the two base change maps. Then $\pi^{t}$ also lies in the image of that base change map.
\end{proposition}
\begin{proof}
By Proposition \ref{equiv} we get that $\pi^{\vee}\cong \pi^{\tau}$. Let $\rho=\rec(\pi)$. By \Cref{rmk:parity_gen}, there exists a conjugate self-dual $W'_E$-representation $\rho''$ that appears in the decomposition \eqref{eq:rho_factor} of $\rho$ with odd multiplicity. Write $\rho''=\rho_0\otimes \Sp(m)$ where $\rho_{0}$ is conjugate self-dual. 

Let $\rho^t=\rec(\pi^t)$. By \Cref{zi_prop}, $\rho^t$ is also conjugate self-dual and has a similar decomposition 
\[
\rho^t = \bigoplus_{1\leq i \leq r_{t}}\left(\rho'^t_{i}\oplus ((\rho'^t_i)^{\tau})^{\vee} \right)\oplus \bigoplus_{1\leq j \leq s_{t}}\rho''^t_{j}.
\]
If $s_t=0$, then $\rho^t$ is in the image both of stable and unstable base change and we have the proposition. Thus assume that $s_{t}\geq 1$. Let $\rho''^{t}_{1}=\rho_{0}^{t}\otimes \Sp(m_{t})$. Since the M\oe{}glin-Waldspurger algorithm to obtain Zelevinsky involution of a rigid representation preserves the cuspidal line it's supported in, we conclude that $\rho^{t}_{0}=\rho_{0}$ and $m-m_{t}$ is even. Thus $\rho''^{t}_{1}$ is of the same parity as $\rho''$. Applying \Cref{rmk:parity_gen} once again we obtain the proposition.
\end{proof} 

\section{Relationship between distinction and base change for ladders}\label{s_dist_bc_ladder}
\subsubsection{Ladder representations}\label{sss: ladder}
\begin{definition}\label{def_lad}
  Let $\sigma\in \CuspE$. The multi-set $\{\Delta_1,\dots,\Delta_k\}\in\cO_\sigma$ is called a \emph{ladder} if
\begin{equation*}
  b(\Delta_1)>\cdots>b(\Delta_k)\quad  
  \text{and}\quad e(\Delta_1)>\cdots>e(\Delta_k).
\end{equation*}
A representation $\pi\in \IrrE$ is called a ladder representation if
$\pi=L(\fmm)$ where $\fmm\in \cO_\sigma$ is a ladder.
\end{definition}
Whenever we say that $\fmm=\{\Delta_1,\dots,\Delta_k\}\in\cO_\rho$ is a ladder,
we implicitly assume that $\fmm$ is already ordered as in the definition above.
\begin{remark}
The M\oe{}glin-Waldspurger algorithm to compute the Zelevinsky involution takes a particularly simple form if $\fmm$ is a ladder, as described in \cite[\S 3.2]{LM}.
\end{remark}
\begin{definition}\label{def: prop l}
  A ladder, $\fmm=\{\Delta_1,\dots,\Delta_k\}\in\cO_\sigma$ is called a
  \emph{proper ladder} if $\Delta_{i+1}\prec\Delta_i$, $i=1,\dots, k-1$. If
  $\fmm$ is a proper ladder then $L(\fmm)$ is called a proper ladder
  representation.
\end{definition}
Note that if $\pi$ is a ladder representation, then it can be decomposed as
$\pi=\pi_{1}\times \cdots\times \pi_{k}$ where each $\pi_{i}$ is a proper ladder
representation. This decomposition is unique up to a permutation of the
$\pi_{i}$'s appearing in the product.

\subsection{Conjecture \ref{jac2} in the ladder case}\label{ss_ladder_dis}
Next we recall the main results of \cite{G} concerning the class of ladder representations.
\subsubsection{Definition of $\gamma$} \label{sec:def_gamma}
Before we state them we need to introduce some more notation. Let $\pi\in \IrrE$ such that $\pi^{\vee}\cong \pi^{\tau}$. Suppose that $\supp(\pi)\subseteq \sigma^{\Z}$ for some $\sigma \in \CuspE$. By the argument in the first paragraph of the proof of \Cref{parity}, there exists $\sigma'\in \sigma^{\Z}\sqcup (\nu^{\frac{1}{2}}\sigma)^{\Z}$ such that $(\sigma')^{\vee}\cong (\sigma')^{\tau}$. Note that both the cuspidal lines
$\sigma^{\Z}$ and $ (\nu^{\frac{1}{2}}\sigma)^{\Z}$ are invariant under the
action $\bullet \mapsto (\bullet)^{\tau\vee} $ of twisting by $\tau$ and then taking contragredient. 

  By \cite[Theorem 7]{K}, there exists an integer $a_{\sigma'}\in \{0,1\}$ such
  that $\sigma'$ is $(H,\omega^{a_{\sigma'}})$-distinguished. This integer is
  unique by \cite[Corollary 1.6]{AKT}. 

 We then make the following definition:
\begin{defn}
  Let $\sigma\in \CuspE$ and suppose that there exists
  $\sigma'\in \sigma^{\Z}\sqcup (\nu^{\frac{1}{2}}\sigma)^{\Z}$ such that
  $(\sigma')^{\vee}\cong (\sigma')^{\tau}$. Define
  $\gamma(\sigma)=\gamma(\sigma^{\Z})\in \{0,1\}$ to be
\begin{equation*}
\gamma(\sigma)= \Bigg\{ \begin{array}{ll}
a_{\sigma'}& {\rm if}\  \sigma'\in \sigma^{\Z}\\
1-a_{\sigma'} & {\rm if}\ \sigma' \in (\nu^{\frac{1}{2}}\sigma)^{\Z}.
\end{array}
\end{equation*}
\end{defn}
\subsubsection{Distinguished ladders}
We recollect \cite[Theorem~4.2 and Theorem~4.3]{G}:
\begin{theorem}\label{j_ladder}
  Let $\pi=L(\fmm)\in \IrrE$ be a ladder representation such that
  $\supp(\pi)\in \sigma^{\Z}$ and $\pi^{\tau}\cong \pi^{\vee}$. Write
  $\pi=\pi_{1}\times \cdots\times \pi_{k}$ with each $\pi_{i}$ a proper ladder
  representation. Let $\abs{\fmm}=t$. Then we have:
\begin{enumerate}
\item If $k$ is even, then $\pi$ is both $H$-distinguished and
  $(H,\omega)$-distinguished.\\
\item In $k$ is odd, $\pi$ is
  $(H,\omega^{\gamma(\sigma)+t+1})$-distinguished. Moreover, it cannot be both
  $H$-distinguished and $(H,\omega)$-distinguished.
\end{enumerate}
In particular, Conjecture \ref{jac2} holds for the class of ladder
representations.  
\end{theorem}

\subsection{ \Cref{rafli2} in the ladder case}
The fact that Conjecture \ref{rafli2} holds in the ladder case is an immediate consequence of Proposition \ref{equiv} and Theorem \ref{j_ladder}.  We now discuss some finer results for this class of representations.
\subsubsection{}\label{sss:ladder_t}
We retain the notation introduced in \Cref{sec:def_gamma}. Let $\pi=L(\fmm)$ such that $|\fmm|=t$. Write
\begin{equation*}
\fmm= \{[\nu^{a_{1}'}\sigma',\nu^{b_{1}'}\sigma'],\dots,[\nu^{a_{t}'}\sigma',\nu^{b_{t}'}\sigma']\}
\end{equation*}
where $a_{i}'>a_{i+1}', b_{i}'>b_{i+1}'$ for all $i=1,\dots, t-1$, and $\{a_{i}',b_{i}'\mid i=1,\dots,t\}$ is either a subset of $(\frac{1}{2}+\mathbb Z)$ or of $\mathbb Z$.
\begin{lemma}\label{t_even}
The representation $\pi$ described above lies in the intersection of both the base change maps if and only if $t$ is even.
\end{lemma}
\begin{proof}
By Proposition \ref{equiv}, $\pi$ lies in the union of the images of the two base change maps. Let $\rho_{i}=\rec(L([\nu^{a_{i}'}\sigma',\nu^{b_{i}'}\sigma']))$. Thus $\rec(\pi)=\rho_{1}\oplus\cdots \oplus \rho_{t}$. By the definition of ladder representations, $\rho_{i}\ncong \rho_{j}$ if $i\neq j$. Moreover, since $\fmm^{\vee}\cong \fmm^{\tau}$, $(\rho_{i}^{\vee})^{\tau}\cong \rho_{t-i+1}$ for all $i$ and $\rho_{i}$ is conjugate self-dual if and only if $t$ is odd and $i=\frac{t+1}{2}$. \Cref{rmk:parity_gen} now gives the lemma.
\end{proof}

\subsubsection{} We get the following refined result in the ladder case using
\Cref{t_even} and \Cref{j_ladder}.
\begin{theorem}\label{rf_ladder}
Let $\pi=L(\fmm)\in \IrrE$ be a ladder representation where $\fmm \in \mathcal O_{\sigma}$. Let $|\fmm|=t$. Further let $\pi=\pi_{1}\times \cdots\times \pi_{k}$, where each $\pi_{i}$ is a proper ladder representation. Then we have:
\begin{enumerate}
\item In case $k$ is even, then for any $a\in \{0,1\}$, $\pi$ is $(H,\omega^{a})$-distinguished if and only if $\pi$ is in the intersection of the image of the two base change maps.\\
\item In case $k$ is odd and $t$ is even, $\pi$ is $(H,\omega^{\gamma(\sigma)+1})$-distinguished if and only if $\pi$ is in the image of either of the two base change maps. Also, in this case $\pi$ cannot be $(H,\omega^{\gamma(\sigma)})$-distinguished.\\
\item In case both $k$ and $t$ are odd, $\pi$ is $(H,\omega^{\gamma(\sigma)})$-distinguished if and only if $\pi$ is in the image of exactly one of the two base change maps, the type of which is determined by $\sigma$. Also, in this case $\pi$ cannot be $(H,\omega^{\gamma(\sigma)+1})$-distinguished.\\
\end{enumerate}
In particular, \Cref{rafli2} holds for the class of ladder representations.
\end{theorem}
\begin{proof}
  We will prove (1). If $\pi$ is $(H,\omega^{a})$-distinguished, then we have
  $\pi^{\tau}\cong \pi^{\vee}$ (by \cite[Proposition
  12]{F}). Note that the fact $k$ is even forces $t$ to be even as well. The `only if' part of (1) now follows from \Cref{t_even}.

  For the other direction, suppose that $\pi$ is in the intersection of the
  images of the base change maps. Appealing to \Cref{equiv} and part (1) of
  Theorem \ref{j_ladder}, we get the statement.

  Parts (2) and (3) of the theorem can be proved in a similar manner. We omit
  the details.
\end{proof}

\section{The case of representations induced from ladders}\label{s_ind_ladder}
Now that we have the complete picture regarding distinction and its relation to the two base change maps for the class of ladder representations, we explore it further for the class of representations that are irreducibly induced from ladder representations.
\subsubsection{} The structure of a representation lying in this class can be described explicitly using \cite[Lemma 5.17, Lemma 5.21]{LM1}, which together give a combinatorial criterion to determine exactly when a representation induced from ladders is irreducible.

\subsection{Distinction in the mutually unlinked case}

\subsubsection{Geometrical lemma in this setup}\label{gl_mu}
Let $\pi_{i}\in \IrrE$ and set $\pi=\pi_{1}\times\dots\times \pi_{k}$ and $\la=\pi_{1}\otimes\dots\otimes \pi_{k}$. Let $M$ and $P$ be the standard Levi subgroup and the standard parabolic subgroup of $G_{n}$ corresponding to $\la$ respectively. The representations $\pi|_{H_{n}}$ is glued from representations that are parameterized by the orbits of the natural action of $P$ on the symmetric space $G_{n} / H_{n}$. Before we proceed further, we will briefly recall some preliminaries on the representatives and orbits corresponding to the double coset space $P\setminus G_{n} / H_{n}$. Our purpose here is to merely set up some notation that will be used in the proof of \Cref{conv_her_prop}. For details we refer the reader to \cite{JLR}.

Let $W$ and $W_{M}$ be the Weyl group of $G_{n}$ and $M$ respectively. Denote by $W[M]$ be the subset of representatives of the double coset $W_{M}\setminus W / W_{M}$ those are of minimal length and are of order at most two. The set $W[M]$ is in bijection with the double coset $P\setminus G_{n} / H_{n}$ (\cite[Proposition 20]{JLR}). For a $w\in W[M]$, the group $M(w):=M\cap wMw^{-1}$ is a standard Levi subgroup of $G_{n}$ contained in $M$. Thus we can enumerate the blocks of $M(w)$ in the following way:
\begin{equation}\label{eq:def_JM}
\fJ_{M(w)} \cong \fJ_{M(w),M} = \set{(i,j)|  i=1,\ldots,
k,\quad j=1,\ldots, s_i}.
\end{equation}
Clearly, for a block of the Levi subgroup $M(w)$, the first index of an element
of $\fJ_{M(w)}$ corresponds to the block of $M$ it sits in. We order
$\fJ_{M(w)}$ in the Lexicographic ordering. The condition that $w$ is
an involution and $M(w)$ is a standard Levi subgroup of $M$, induces an
involution $\epsilon_{w}$ on $\fJ_{M(w)}$ which satisfies the following
relation. For $1\leq i\leq k$ and $1\leq j_{1}< j_{2}\leq s_{i}$,
\begin{equation}\label{eq_gl_mu}
\epsilon_{w}(i,j_{1})=(i_{1}',j_{1}'), \epsilon_{w}(i,j_{2})=(i_{2}',j_{2}') \Longrightarrow  i_{1}'< i_{2}'.
\end{equation}

The only fact that we need for our purpose is as follows. If $\pi$ is
$H$-distinguished, then there exist a $w_{1}\in W[M]$, a Levi subgroup $M(w_1)$
of $M$, an involution $\epsilon_{w_1}$ on $\fJ_{M(w_1)}$, and an
irreducible component of the Jacquet module $r_{M(w_{1}),M}(\la)$, which we
denote by $\otimes_{(i,j)\in \fJ_{M(w_1)}} \sigma_{(i,j)}$, such that
$\sigma_{\epsilon_{w_1}(i,j)}\cong (\sigma^\vee_{(i,j)})^{\tau}$ for all
$(i,j)\in \fJ_{M(w_1)}$ and $\sigma_{(i,j)}$ is $H$-distinguished when
$\epsilon_{w_1}(i,j)=(i,j)$.

\subsubsection{Mutually unlinked representations}\label{sec:mut_unlink}
\begin{defn}\label{def_mut_unl}
We will say that $\pi_1= L(\gotM_1), \ldots, \pi_k= L(\gotM_k)\in \IrrE$ are mutually unlinked representations, if for every $i\neq j$, $\Delta\in \gotM_i$ and $\Delta'\in \gotM_j$, the pair of segments $\{\Delta, \Delta'\}$ is not linked.
\end{defn}
When $\pi_1, \ldots, \pi_k$ are mutually unlinked, $\pi=\pi_1 \times\cdots \times \pi_k$ is irreducible (see \cite[Proposition 8.5]{Z}) and $\pi\cong \pi_{w(1)}\times \cdots \times \pi_{w(k)}$ for every permutation $w\in \mathcal S_k$ (see \cite[Theorem 1.9]{Z}).

\subsubsection{} \Cref{conv_her_prop} provides a converse to the hereditary property of $H$-distinction for the representations in $\IrrE$ that are induced from ladders satisfying an additional hypothesis that they are mutually unlinked. To that end we need the
following auxiliary lemma.
\begin{lemma}\label{auxlem}
Suppose that $\pi\in \Irr_{E}$ is a proper ladder representation and let $\pi_1\otimes \pi_2$ be an irreducible component of the Jacquet module of $\pi$ with respect to the corresponding maximal parabolic. Further suppose that $\pi_1= L(\gotM_1)$ for a multisegment $\gotM_1$ such that $\gotM_1$ is a proper ladder and contains at least two segments. Then $\pi$ and $\pi_1$ are not mutually unlinked.
\end{lemma}
\begin{proof}
Suppose $\pi= L(\gotM)$ and that $\Delta'\prec \Delta$ are two segments in $\gotM_1$. By \cite[Theorem 2.1]{KL}, $\exists \ \Delta''\in \gotM$ such that $e(\Delta'')=e(\Delta')$ and $b(\Delta'')\leq b(\Delta')$. Hence $\Delta''\prec \Delta$.
\end{proof}

\subsubsection{}\label{sec:pf_h}
\begin{proposition}\label{conv_her_prop}
Suppose that $\pi_1,\ldots,\pi_k\in \Irr_{E}$ are mutually unlinked proper ladder representations. Then $\pi= \pi_1\times\cdots\times \pi_k\in \Irr_{E}$ is $H$-distinguished if and only if there is a permutation $w\in \mathcal S_k$ such that:
\begin{enumerate}
\item $\pi_{w(i)}\cong (\pi_i^\vee)^{\tau}$ for all $1\leq i\leq k$
\item For $i$ such that $w(i)=i$, $\pi_i$ is $H$-distinguished.
\end{enumerate}
\end{proposition}

\begin{proof}
Suppose that a permutation such as in the statement exists. Renumbering the $\pi_{i}$'s if necessary, we get that
\begin{equation*}
  \pi\cong \pi_1\times \cdots \times \pi_{i_0} \times
  (\pi_{i_0+1} \times \cdots \times \pi_{k-i_0}) \times 
  (\pi_{i_0}^\vee)^{\tau}\times\cdots \times (\pi_1^\vee)^{\tau}
\end{equation*}
where $\pi_{i_0+1},\ldots, \pi_{k-i_0}$
are all $H$-distinguished. Hence $\pi_{i_0+1} \times \cdots \times \pi_{k-i_0}$
is $H$-distinguished (\cite[Proposition 26]{F1}) and the existence of a
non-trivial $H$-invariant functional on $\pi$ then follows from \cite[Theorem
2.8]{BD}.

\subsubsection{}\label{sec:pf_ch1}

Conversely, suppose that $\pi$ is an $H$-distinguished representation. Then, by
\cite[Proposition 12]{F}, we have $(\pi^\vee)^{\tau}\cong \pi$. Write
$\pi_i = L(\gotM_i)$ for all $i$. Then $\pi= L(\gotM_1+\cdots + \gotM_k)$, which
means
\begin{equation*}
(\gotM_1^\vee)^{\tau}+\cdots + (\gotM_k^\vee)^{\tau} = \gotM_1+ \cdots + \gotM_k.
\end{equation*}
In particular, for each $i$,
\begin{equation*}
\gotM_1 + \cdots + \gotM_k \geq (\gotM_{i}^{\vee})^{\tau}.
\end{equation*}
Since all pairs of segments in $\gotM_i$ are linked, we must have
$(\gotM_i^\vee)^{\tau}\leq \gotM_{w(i)}$ for some $1\leq w(i)\leq k$. It now
easily follows that $w$ can chosen to be an involution on $\{1,\ldots,k\}$ for
which $(\gotM_i^\vee)^{\tau} = \gotM_{w(i)}$. The first condition in the
statement is then clear.

Suppose that $i_0$ is a fixed point of $w$. It remains to be shown that
$\pi_{i_0}$ is an $H$-distinguished representation.

\subsubsection{} \label{sec:pf_ch2} 

Let us write $I\subset \{1,\ldots,k\}$ for the subset of indices $i$ for which $\pi_i\cong \pi_{i_0}$. Note that if $|I|$
is even, we can choose a permutation $w$ which satisfies Condition (1) on $I$
instead of being the identity. Thus in this case there is nothing to prove. So
assume that $|I|$ is odd.

We also define the following additional subsets of indices in $\{1,\ldots,k\}$:
\[
J = \set{i\in \set{1,\ldots,k}\setminus I| (\gotM_i^\vee)^{\tau} = \gotM_i,
\; \ |\gotM_i|>1}
\]
and
\[
K = \begin{cases}
\set{ i\in \{1,\ldots,k\}\setminus (I\cup J)\;\mid
    \; \exists \Delta\in \fmm_i \ \text{s.t. } \Delta_0\subset \Delta} &
    \text{if } \gotM_{i_0}=\{\Delta_0\} \\
    \emptyset & \text{if }    |\gotM_{i_0}|>1
  \end{cases}.
\]
Set $J'=J\cup I$ if $|\fmm_{i_{0}}|>1$ and $J'=J$ otherwise. Note that $J'$ is precisely the set of indices in $I\cup J$ such that $|\fmm_{i}|>1$. By
rearranging the $\pi_{i}$'s we can assume that $J = \set{1,\ldots,k_0}$,
$I = \{k_0+1,\ldots, k_1\}$ and $K = \set{k_1+1,\ldots, k_2}$ for some
$1\leq k_0 < k_1\leq k_2\leq k$.

\subsubsection{}\label{ss_conv_her_prop1}
Let $M$ be the Levi subgroup corresponding to
$\pi_{1}\otimes\cdots \otimes \pi_{k}$ and let $w_{1}$, $\epsilon_{w_{1}}$ and
$\fJ_{M(w_1)}$ be as defined in \S\ref{gl_mu}. Define
\[
L = \set{i\in I\cup J\,\mid\, s_i=1}.
\]

\begin{claim}\label{claim_ss_conv_her_prop1}
If $i\in L$ and $k_1< r \leq k$, then $\epsilon_{w_1}(i,1)\neq (r,1)$.
\end{claim}

Assume, if possible, the contrary for some $i\in L$ and $k_1<r \leq k$. Then we
have $\pi_{i}\cong (\pi_{i}^\vee)^{\tau}\cong \sigma_{(r,1)}$. If $i\in J'$,
then appealing to Lemma \ref{auxlem} we obtain that $\pi_{i}$ and $\pi_{r}$ are
not mutually unlinked which contradicts the hypothesis of the
statement. Therefore let us assume that $J'\neq I\cup J$ (i.e. $|\gotM_{i_0}|=1$), and $i\in I$. Thus 
$\gotM_{i}=\{\Delta_{0}\}$. By assumption $\Delta_{0}=(\Delta_{0}^{\vee})^{\tau}$. Since
$\sigma_{(r,1)}=L(\Delta_{0})$, by \cite[Theorem 2.1]{KL}, there is a segment
$\Delta'\in \gotM_{r}$ for which $e(\Delta')=e(\Delta_{0})$ and
$\Delta_{0} \subset \Delta'$.

Suppose first that $|\gotM_r| =1$. Since $r\not\in I$, $\Delta'\neq \Delta_{0}$ and
thus $\Delta_{0} \subsetneq \Delta'$. In particular,
$\Delta'\prec ((\Delta')^\vee)^{\tau}\in (\gotM_r^\vee)^{\tau} =
\gotM_{w(r)}$.
Clearly $\gotM_{w(r)}=\{((\Delta')^\vee)^{\tau}\}\neq \gotM_{r}$. This
contradicts the hypothesis that $\fmm_{r}$ and $\fmm_{w(r)}$ are mutually
unlinked.

We are left with the case when $|\gotM_r| >1$. Since $\pi_r$ is a proper ladder,
if we had $\Delta'=\Delta_{0}$, it would mean there is a segment in $\gotM_r$ linked
with $\Delta_{0}$. Hence, $\Delta_{0} \subsetneq \Delta'$. Now, since $r\not\in J$,
$\gotM_{w(r)}=(\gotM_r^\vee)^{\tau}\neq \gotM_r$ and arguing as above with
$\Delta'$ and $((\Delta')^\vee)^{\tau}$ gives a contradiction.

Thus we have demonstrated the claim.

\subsubsection{}\label{ss_conv_her_prop2}
\begin{claim}\label{claim_ss_conv_her_prop2}
$L=I\cup J$.
\end{claim}
We will prove it by contradiction. Assume,
if possible, that there exists an index in $\{1,\ldots,k_1\}$ for which
$s_{i_1}>1$ and take $i_1$ to be the minimal such index. Let the irreducible
representation $\otimes_{h=1}^{s_{i_1}} \sigma_{(i_1,h)}$ be a component of the
Jacquet module of $\pi_{i_1}$ with respect to a suitable parabolic
subgroup. As supposed earlier for a ladder multi-set, write $\gotM_{i_1} = \Delta_1 + \cdots + \Delta_n$, in a manner such that $\Delta_{j+1} \prec \Delta_j$ for all $1\leq j\leq n-1$. Since
$(\gotM_{i_1}^\vee)^{\tau} = \gotM_{i_1}$, there exists $\sigma' \in \CuspE$
such that $(\sigma')^{\vee}\cong (\sigma')^{\tau}$ and $\gotM_{i_1}$ is
supported on $(\sigma')^{\mathbb Z}$ or
$(\nu^{\frac{1}{2}}\sigma')^{\mathbb Z}$. Also we have
$(\Delta_j^\vee)^{\tau} = \Delta_{n+1-j}$ for all $j$. By \cite[Theorem
2.1]{KL},
\begin{equation*}
\sigma_{(i_1, h)} = L( \sum_{j=1}^n [\nu^{c^h_j}\sigma',\nu^{(c^{h-1}_j-1)}\sigma'])
\end{equation*}
are ladder representations with $c^h_1 > \ldots > c^h_n$ for all $h$ and
$c^{s_{i_1}}_j\leq \ldots\leq c^0_j$ for all $j$. Here
$\{c_{j}^{h}\mid \forall\ h,j\}$ is either a subset of
$ (\frac{1}{2}+\mathbb Z)$ or that of $\Z$, and $c^h_j-c^{h-1}_j\in \mathbb Z$
for all $j,h$. Moreover\footnote{Note that there is no dissonance here, because
  the exponents $c^{h}_{j}$ that we are considering possibly lie in
  $\frac{1}{2}+\Z$, with the definition of segments given in \S
  \ref{s_irr_gln}. Since $c^h_j-c^{h-1}_j+1$ is still an integer, this simply
  means that in that case the segments defining $\sigma_{(i,h)}$ are supported
  on $(\nu^{\frac{1}{2}}\sigma')^{\Z}$ instead of $(\sigma')^{\Z}$.},
$\nu^{(c^0_j -1)}\sigma'= e(\Delta_j)$ and
$(\nu^{c^{s_{i_1}}_j})\sigma'= b(\Delta_j)$.

Let us write $\epsilon_{w_1}(i_1,h) = (\alpha(h), \rho(h))$ for all
$1\leq h\leq s_{i_1}$. Note that by (\ref{eq_gl_mu}) $\alpha$ is injective and
increasing.

We first show that $\alpha(1)> k_{1}$. Assume, if possible, that
$\alpha(1) \leq k_{1}$. By the relation in (\ref{eq_gl_mu}) and the minimality
of $i_{1}$, we get that $\rho(1)=1$. Thus
$\epsilon_{w_1}(i_1,1) = (\alpha(1), 1)$ and
\begin{equation}\label{conv_her_prop2}
((\sigma_{(i_{1},1)})^{\vee})^{\tau}\cong \sigma_{(\alpha(1),1)}.
\end{equation}
Since both $\pi_{i_{1}}$ and $\pi_{\alpha(1)}$ are conjugate self-dual and
$s_{i_{1}}>1$, it is easy to see that the isomorphism in (\ref{conv_her_prop2})
cannot exist (say by matching the central characters of the two
representations). Thus indeed $\alpha(1)> k_{1}$. Since $\alpha$ is increasing,
$\alpha(h)> k_{1}$ for all $h$.

From minimality of $i_1$, we have $i\in L$ for all $1\leq i< i_1$. It then
follows from  \Cref{claim_ss_conv_her_prop1} that $\rho(h)=1$ for all $h$.

Since $\pi_{\alpha(h)}$ is a ladder, we know that
$\nu^{-c^h_1}\sigma' = e(\Delta^h)$ for some $\Delta^h\in
\gotM_{\alpha(h)}$.
Let $\Delta^{h}=[\nu^{a_{h}}\sigma', \nu^{b_{h}}\sigma']$ where
$\{a_{h},b_{h}\mid h=1,\dots, s_{i_{1}}\}$ is either a subset of
$ (\frac{1}{2}+\mathbb Z)$ or that of $\Z$. Also $1-c^{h-1}_1 \geq a_{h}$. In
particular, when $c^h_1<c^{h-1}_1$ the segment $\Delta^h$ is non-trivial.  For
the cases of $c^h_1=c^{h-1}_1$ we will take $\Delta^h$ to be the trivial
segment.

Let us write all indices $1\leq h_1< \ldots < h_t\leq s_{i_1}$ for which
$\Delta^{h_i}$ is non-trivial, arranged by $b_{h_1}< \ldots < b_{h_t}$. Set
$h_{0}=0$. Note that since $\Delta^{h_i}$ ($i\geq 1$) is non-trivial, we have
$c_{1}^{h_{i}-1}=c_{1}^{h_{i-1}}$ and so
\begin{equation}\label{conv_her_prop3}
a_{h_{i}}-1\leq -c_{1}^{h_{i-1}}.
\end{equation}
Since $\alpha$ is injective and $\pi_i$'s are mutually unlinked, no two of the
segments $\Delta^{h_i}$'s should be linked. Therefore we must have
$a_{h_{t}}\leq \ldots \leq a_{h_{1}}$.

Observe that
$e(\Delta^{h_t})= \nu^{-c_{1}^{h_{t}}}\sigma'= \nu^{-c_{1}^{s_{i_{1}}}}\sigma'=
(b(\Delta_{1})^{\vee})^{\tau}=e(\Delta_n)$.
Moreover, $a_{h_{t}}\leq a_{h_{1}}\leq -c^0_1+1$, and
$b(\Delta_{n})= \nu^{-c_{1}^{0}+1}\sigma'$. Thus $\Delta^{h_t}$ contains
$\Delta_{n}$ and has the same end.

In case $n>1$, since $\Delta_n\prec \Delta_{n-1}$, we get that
$\Delta^{h_t}\prec \Delta_{n-1}$. This is a contradiction to the fact that
$\pi_{i_1}$ and $\pi_{\alpha(h_t)}$ are mutually unlinked. This forces $n$ to be
equal to 1.

If $n=1$, then $i_1\not\in J$, and hence,
$\gotM_{i_1}= \gotM_{i_0}=\{\Delta_{0}\}$. Let
$\Delta_{0}=[\nu^{-a}\sigma',\nu^{a}\sigma']$ ($a\in \frac{1}{2}\Z$). Arguing as
above with $\Delta_{0}=\Delta_1 = \Delta_n$, we get that
$\Delta_{0}\subset \Delta^{h_t}$ and $a=b_{h_t}$. Our arrangement of the indices
forces $\alpha(h_t)\leq k_2$. Since $\alpha$ is increasing,
$\alpha(h_1)\leq k_2$ and thus there exists a segment $\Delta'\in \gotM_{h_1}$
such that $\Delta_{0}\subset \Delta'$. By (\ref{conv_her_prop3})
$a_{h_{1}}\leq -a$. If $t>1$, $b_{h_{1}}< b_{h_{t}}= a$. By definition of
$\Delta^{h_{1}}$, $b_{h_{1}}\geq -a-1$. Since $\Delta^{h_1}$ and $ \Delta_{0}$
can't be linked, it follows that $a_{h_1} = -a$. Hence,
$\Delta^{h_1}\subsetneq \Delta_{0} \subset \Delta'$, which gives a contradiction
to the fact that $\gotM_{h_1}$ is a ladder. Thus $t=1$. The condition $n=t=1$
forces $s_{i_1}=1$.

This contradicts the assumption that there is an index $i_{1} \leq k_{1}$ such
that $s_{i_{1}}>1$. This proves our claim that $L=I\cup J$.

\subsubsection{}\label{sec:pf_ch5}
Let $1\leq i\leq k_{1}$. Then by \Cref{claim_ss_conv_her_prop2} we
have that $s_{i}=1$. Let $\epsilon_{w_1}(i,1)= (r,r')$. It now follows from
\Cref{claim_ss_conv_her_prop1} and \Cref{claim_ss_conv_her_prop2} that $r\leq k_{1}$ and
$r'=1$. Thus $\epsilon_{w_1}$ induces a permutation on $I\cup J$ which clearly
preserves the set $I$. Since by assumption $\abs{I}$ is odd, there exists an $i'\in I$ such that $\epsilon_{w_1}(i',1)= (i',1)$. Thus we get that $ \pi_{i'}$ is $H$-distinguished, and since $\pi_{i_{0}}\cong \pi_{i'}$, this finishes the proof of the converse. 
\end{proof}

\subsubsection{}
A simple consequence of the above statement in the following:
\begin{remark}
Proposition \ref{conv_her_prop} provides a host of examples of conjugate self-dual rigid representations that are neither $H$-distinguished nor $(H,\omega)$-distinguished, the simplest one being $L([\nu^{-\frac{1}{2}}],[\nu^{\frac{1}{2}}])\times L([\nu^{-\frac{1}{2}},\nu^{\frac{1}{2}}])$. Thus we obtain a family of counterexamples to Conjecture \ref{jac2} and and hence of Conjecture \ref{rafli2} (using Proposition \ref{equiv}).
\end{remark}

\section{An imprimitive counterexample}\label{sec_imp_ce}
Recall that an irreducible representation of $G_{n}$ is said to be \textit{imprimitive} if it is not parabolically induced from a representation of a proper Levi subgroup. For instance it is known that the set of proper ladder representations are precisely the imprimitive ladder representations (\cite[Theorem 16]{LM}). In view of our results on the proper ladder representations, it is natural to wonder if Conjecture \ref{rafli2} and Conjecture \ref{jac2} might
hold for this smaller class of irreducible representations. We now provide a counterexample demonstrating the absence of such a possibility.
\subsubsection{} We consider an irreducible representation $\theta$ defined as the following:
\begin{defn}\label{imprm_ce}
Let $\sigma\in \CuspE$ be a conjugate self-dual representation. Set
\begin{equation*}
  \fmm=\{[\nu^{\frac{1}{2}}\sigma,\nu^{\frac{3}{2}}\sigma],[\nu^{-\frac{1}{2}}\sigma,\nu^{\frac{7}{2}}\sigma],[\nu^{-\frac{3}{2}}\sigma,\nu^{-\frac{1}{2}}\sigma],[\nu^{-\frac{5}{2}}\sigma,\nu^{\frac{5}{2}}\sigma],[\nu^{-\frac{7}{2}}\sigma,\nu^{\frac{1}{2}}\sigma]\} 
\end{equation*}
and 
\begin{equation*}
\theta=L(\fmm).
\end{equation*}
\end{defn}

\subsubsection{} 
Clearly $\theta$ is conjugate self-dual.  We claim that $\theta$ is imprimitive. Write $\fmm=\fmm_{1}+\fmm_{2}$ where $|\fmm_{1}|\leq 2$. It can be easily seen using \cite[Proposition 4.15(3)]{LM1} or \cite[Lemma 5.21]{LM1} (depending on whether $|\fmm_{1}|$ is $1$ or $2$) that $L(\fmm_{1})\times L(\fmm_{2})$ is always reducible. Thus the representation $\theta$ is imprimitive. 

\subsubsection{} Our claim now follows from the next lemma:
\begin{lemma}\label{lem_imp_ce}
The representation $\theta$ is neither $H$-distinguished nor $(H,\omega)$-distinguished.
\end{lemma}
\begin{proof}
  Recall that if $\pi\in \IrrE$ is $H$-distinguished and non-generic, then
  $\nu^{\frac{1}{2}}\pi'$ must be $H$-distinguished for at least one irreducible
  subquotient $\pi'$ of one of the derivatives of $\pi$ (see \cite[Lemma
  2.4]{AKT}). Let
\begin{equation*}
  \fmm_{1}=\set{[\nu^{-\frac{1}{2}}\sigma,\nu^{\frac{7}{2}}\sigma],[\nu^{-\frac{5}{2}}
    \sigma,\nu^{\frac{5}{2}}\sigma],[\nu^{-\frac{7}{2}}\sigma,\nu^{\frac{1}{2}}\sigma]}
  \quad
  \text{and} \quad  \fmm_{2}=\set{[\nu^{\frac{1}{2}}\sigma,
    \nu^{\frac{3}{2}}\sigma],[\nu^{-\frac{3}{2}}\sigma,\nu^{-\frac{1}{2}}\sigma]}.
\end{equation*}
By \cite[Proposition 8.4]{Z} $\theta$ is a subquotient of
$L(\fmm_{1})\times L(\fmm_{2})$. Since $\fmm_{1}$ and $\fmm_{2}$ are ladders,
using the Langlands classification version of \cite[Theorem 14]{LM}, it can be
checked easily that no irreducible subquotient of any derivative of
$L(\fmm_{1})\times L(\fmm_{2})$ when twisted by the character
$\nu^{\frac{1}{2}}$ is $H$-distinguished. Since the process of taking
derivatives is an exact functor, the same is true for $\theta$ and thus it is
not $H$-distinguished.

Since $\sigma$ is an arbitrary conjugate self-dual cuspidal representation, twisting it by $\chi_{-1}$ if necessary, we can conclude that $\theta$ is not $(H,\omega)$-distinguished either.
\end{proof}

\end{document}